\documentclass[11pt]{article}

\usepackage{amsthm}

\newcommand{\bfn}{{\bf n}}

\usepackage{amsmath}
\usepackage{amssymb}
\usepackage{latexsym}
  \newenvironment{proofof}[1]{%
  \noindent {\bf Proof of #1}}%
  {\hspace*{\fill}$\Box$}

  \newtheorem{theorem}{Theorem}[section]
  \newtheorem{lemma} [theorem] {Lemma}
  \newtheorem{conjecture} [theorem] {Conjecture}

\def\enddiscard{}
\long\def\discard#1\enddiscard{}

  \newcommand{\Fn}{F^{\bf n}}
  \newcommand{\KF}{\mass(\cF_t)}
  \newcommand{\KT}{\mass(\cT_t)}

  \newcommand{\pr}{\mathbb P}

  \newcommand{\mass}{\mbox{mass}\,}

  \newcommand{\m}[1]{\marginpar{\tiny{#1}}}

\newcommand{\cU}{{\mathcal U}}
  
  \newcommand{\cA}{{\mathcal A}}
  \newcommand{\cB}{{\mathcal B}}

  \newcommand{\cF}{{\mathcal F}}
  
  \newcommand{\cT}{{\mathcal T}}

  \newcommand{\frag}{{\mbox{\rm frag}}}

\begin{document}
\title{Connectivity for bridge-alterable graph classes}

\author{Colin McDiarmid \\ Department of Statistics, Oxford University\\ 1 South Parks Road, Oxford OX1 3TG, UK\\
cmcd@stats.ox.ac.uk}
\date{26 February 2016}
\maketitle

\begin{abstract}
  A collection $\cA$ of graphs is called {bridge-alterable} if, for each graph $G$ with a bridge $e$, $G$ is in $\cA$ if and only if $G\!-\!e$ is.  For example the class $\cF$ of forests is bridge-alterable.  For a random forest $F_n$ sampled uniformly from the set $\cF_n$ of forests on vertex set $\{1,\ldots,n\}$, a classical result of R\'enyi (1959) shows that the probability that $F_n$ is connected is $e^{-\frac12 +o(1)}$.
  
Recently Addario-Berry, McDiarmid and Reed (2012) and Kang and Panagiotou (2013) independently proved that, given a bridge-alterable class $\cA$, for a random graph $R_n$ sampled uniformly from the graphs in $\cA$ on $\{1,\ldots,n\}$, 
the probability that $R_n$ is connected is at least $e^{-\frac12 +o(1)}$.
Here we give a more straightforward proof, and obtain a stronger non-asymptotic form of this result, which compares the probability to that for a random forest.
We see that the probability that $R_n$ is connected is at least the minimum over $\frac25 n < t \leq n$ of the probability that $F_t$ is connected. 
\bigskip

\end{abstract}


\section{Introduction} \label{sec.intro}

A collection $\cA$ of graphs is {\em bridge-addable} if for each graph $G$ in $\cA$ and pair of vertices $u$ and $v$ in different components, the graph $G+uv$ obtained by adding the edge (bridge) $uv$ is also in $\cA$; that is, if $\cA$ is closed under adding bridges. This property was introduced in~\cite{msw05} (under the name `weakly addable'). If also $\cA$ is closed under deleting bridges we call $\cA$ {\em bridge-alterable}.  Thus
$\cA$ is bridge-alterable exactly when, for each graph $G$ with a bridge $e$, $G$ is in $\cA$ if and only if $G\!-\!e$ is in $\cA$.  The class $\cF$ of forests is bridge-alterable, as for example is the class of series-parallel graphs, the class of planar graphs, and indeed the class of graphs embeddable in any given surface.  All natural examples of bridge-addable classes seem to satisfy the stronger condition of being bridge-alterable.

Given a class $\cA$ of graphs we let $\cA_n$ denote the set of graphs in $\cA$ on vertex set $[n]:=\{1,\ldots,n\}$.
Also, we use the notation $R_n \in_u \cA$ to mean that $R_n$ is a random graph sampled uniformly from $\cA_n$
(where we assume implicitly that $\cA_n$ is non-empty). 

For a random forest $F_n \in_u \cF$, a classical result of R\'enyi~\cite{renyi59} from 1959 shows that, as $n \to \infty$
\begin{equation} \label{eqn.forest-conn}
  \pr( F_n \mbox{ is connected}) = e^{-\frac12 +o(1)}.
\end{equation}
In their investigations on random planar graphs, McDiarmid, Steger and~Welsh~\cite{msw05} showed that, when $\cA$ is bridge-addable, for $R_n \in_u \cA$
\begin{equation} \label{eqn.conj-badd}
  \pr( R_n \mbox{ is connected}) \geq e^{-1}.
\end{equation}
It was observed by the same authors~\cite{msw06} in 2006 that the class of forests seems to be the `least connected' bridge-addable class of graphs, and they made the following conjecture.
\begin{conjecture} \label{conj1}
  When $\cA$ is bridge-addable, for $R_n \in_u \cA$
  \[  \pr( R_n \mbox{ is connected}) \geq e^{-\frac12 +o(1)}. \]
\end{conjecture}
\noindent
This conjecture was then strengthened (see Conjecture 1.2 of~\cite{bbg2010}, Conjecture 5.1 of~\cite{amr2012}, or Conjecture 6.2 of~\cite{cmcd2013}) to the following non-asymptotic form.
\begin{conjecture} \label{conj2}
  When $\cA$ is bridge-addable, for $R_n \in_u \cA$
 \[   \pr( R_n \mbox{ is connected}) \geq \pr( F_n \mbox{ is connected}). \]
\end{conjecture}

Early progress was made on Conjecture~\ref{conj1} by Balister, Bollob{\'a}s and Gerke~\cite{bbg2008,bbg2010}; and recently Norin~\cite{norin2013} made further progress, showing that $\pr(R_n \mbox{ is connected}) \geq e^{-\frac23 +o(1)}$. 
Addario-Berry, McDiarmid and Reed (2012) and Kang and Panagiotou (2013) independently proved the following theorem, which establishes the special case of Conjecture~\ref{conj1} when $\cA$ is bridge-alterable.
\begin{theorem} \cite{amr2012,kp2013} \label{thm.old}
Let $\cA$ be a bridge-alterable class of graphs, and let $R_n \in_u \cA$.  Then
\[ \pr(R_n \mbox{ is connected}) \geq e^{-\frac12 +o(1)}.\]
\end{theorem}
\noindent
Here we give a reasonably short and straightforward proof of the following non-asymptotic form of this result,
which together with~(\ref{eqn.forest-conn}) gives Theorem~\ref{thm.old}.  This is a first step towards Conjecture~\ref{conj2}, at least for a bridge-alterable class.

\begin{theorem} \label{thm.newa}
Let $\cA$ be a bridge-alterable class of graphs, let $n$ be a positive integer, 
let $R_n \in_u \cA$, and let $F_t \in_u \cF$ for $t=1,2,\ldots$.  Let $\alpha= 0.4$.  Then
\begin{equation} \label{eqn.thmnewa}
 \pr(R_n \mbox{ is connected}) \geq \min_{\alpha n \leq t \leq n} \pr(F_{t} \mbox{ is connected}).
\end{equation}
\end{theorem}
\noindent
The value $\alpha=0.4$ can be increased towards $\frac12 \, $ : in the final section of the paper we improve it to $0.48 n$, and discuss pushing it up further to $\frac12$.  Conjecture~\ref{conj2} says that we can push $\alpha$ up to 1.

Since this paper was (essentially) completed, the original Conjecture~\ref{conj1} (for bridge-addable rather than bridge-alterable classes) has been fully proved by Chapuy and Perarnau, see~\cite{cp2016}.


\section{Proof of Theorem~\ref{thm.newa}}
\label{sec.proof}

We use two lemmas in the proof.

\begin{lemma} \label{lem.1}
 Let $\cA$ be a bridge-alterable class of graphs,
  let $n$ be a positive integer, let $R_n \in_u \cA$, and let $F_t \in_u \cF$ for $t=1,2,\ldots$. 
  Then
\begin{equation} \label{eqn.gen1}
  \pr(R_n \mbox{ is connected}) \geq \min_{t=1,\ldots,n} \max \{ e^{-\frac{t}{n}},  \pr(F_t \mbox{ is connected})\}.
\end{equation}
\end{lemma}

\begin{lemma} \label{lem.2a}
 Let $\alpha = 0.4$. For each $n=2,3,\ldots$
\[ \pr(F_n \mbox{ is connected}) < e^{-\alpha}.\]
\end{lemma}

\noindent
To deduce Theorem~\ref{thm.newa} from these lemmas, observe that by Lemma~\ref{lem.2a}, for each $1 \leq t \leq \alpha n$
\[ e^{-\frac{t}{n}} \geq e^{-\alpha} \geq \pr(F_n \mbox{ is connected}), \]
and so the right side in~(\ref{eqn.gen1}) is at least the right side in~(\ref{eqn.thmnewa}).
\medskip
  
\begin{proofof} {Lemma~\ref{lem.1}}  \hspace{.1in}
  Our proof initially follows the lines of the proofs in~\cite{amr2012} and~\cite{kp2013}
  of Theorem~\ref{thm.old},
  in that we aim to lower bound the probability of connectedness for the random graph $\Fn$ introduced below.
  Consider a fixed $n \geq 2$.

  Given a graph $G$, let $b(G)$ be the graph obtained by removing all bridges from $G$.
  We say $G$ and $G'$ are {\em equivalent} if $b(G)=b(G')$. 
  This is an equivalence relation on graphs, and if a graph $G$ is in $\cA_n$
  then so is the whole equivalence class $[G]$.  Thus $\cA_n$ is a union of disjoint equivalence classes.
  To prove the lemma we consider an arbitrary (fixed) equivalence class.

  Fix a bridgeless graph $G$ on vertex set $[n]$
  and let $\cB=[G]$.  Let $G$ have $t$ components, with $n_1,\ldots,n_t$ vertices, where $n = \sum_{i=1}^t n_i$.
  We use $\bfn = (n_1,\ldots,n_t)$ to define probabilities.
  First, given a forest $F \in \cF_t$, let
\[
  \mass(F) = \prod_{i=1}^t n_i^{d_F(i)},
\]
  where $d_F(i)$ denotes the degree of vertex $i$ in $F$.
  For  $\cF' \subseteq \cF_t$ let $\mass(\cF') = \sum_{F \in \cF'} \mass(F)$.
  Now let 
\[ \pr(\Fn =F) = \frac{\mass(F)}{\KF} \;\; \mbox{ for each } F \in \cF_t. \] 
  By Lemma 2.3 of~\cite{amr2012}, for a uniformly random element $R^{\cB}$ of $\cB$,
  \[
    \pr(R^{\cB} \mbox{ is connected}) = \pr(\Fn \mbox{ is connected}).
  \]
  Hence to prove the lemma it suffices to consider $\Fn$, and show that
\begin{equation} \label{eqn.RF}
  \pr(\Fn \mbox{ is connected}) \geq \max  \{ e^{-\frac{t}{n}},  \pr(F_t \mbox{ is connected}) \}.
\end{equation}
  To see this, observe that then the probability that $R_n$ is connected is an average of values each at least
  the right side of~(\ref{eqn.RF}) for some $t$, and so it is at least the right side in~(\ref{eqn.gen1}).
\medskip

  The proof of~(\ref{eqn.RF}) breaks into two parts, and the first is standard.  Given a graph $G$, let $\kappa(G)$ denote the number of components.  By Lemma 3.2 of~\cite{amr2012}, for $i=1,\ldots,t-1$
\[ \pr(\kappa(\Fn) = i+1) \leq \frac1{i} \frac{t}{n} \pr(\kappa(\Fn)=i), \]
  and thus
\[ \pr(\kappa(\Fn) = i+1) \leq \frac1{i!} \left(\frac{t}{n}\right)^{i}  \pr(\kappa(\Fn)=1). \]
  Hence
\begin{eqnarray*}
   1= \sum_{i=0}^{t-1} \pr(\kappa(\Fn)\!=\!i\!+\!1) 
 & \leq &
   \sum_{i=0}^{t-1} \frac1{i!} \left(\frac{t}{n}\right)^{i}  \pr(\kappa(\Fn)\!=\!1) < e^{\frac{t}{n}} \cdot  \pr(\kappa(\Fn)\!=\!1)
\end{eqnarray*}
and so $ \pr(\Fn \mbox{ is connected}) > e^{-\frac{t}{n}}$ (as noted at the end of Section 3 of~\cite{amr2012}).
\smallskip

It remains to show that
\begin{equation} \label{eqn.RF2}
  \pr(\Fn \mbox{ is connected}) \geq \pr(F_t \mbox{ is connected}).
\end{equation}
We may assume that $t \geq 2$.
Let $\cT$ be the class of trees. 
Then
\begin{equation}\label{eqn.KT}
  \KT = \prod_{i=1}^t n_i \cdot n^{t-2}.
\end{equation}
This result is proved for example in~\cite{amr2012}
(see the proof of Lemma 4.2)  and in~\cite{kp2013}, though in fact it has long been known, see Theorem 6.1 of Moon~\cite{moon70} (1970), and see also Problems 5.3 and 5.4 of Lov\'asz~\cite{lovasz}. 
  We let $N=\prod_{i=1}^t n_i$ and rewrite~(\ref{eqn.KT}) as
\begin{equation}\label{eqn.KT2}
  \KT = N (\frac{n}{t})^{t-2}\cdot |\cT_t|.
\end{equation}
  For the case $t=2$, $\ \mass(\cT_2)=n_1 n_2 \, $ and
  $\, \mass(\cF_2) = \mass(\cT_2) +1$, so
\[ \pr(\Fn \mbox{ is connected}) = 
\frac{n_1 n_2}{n_1 n_2 +1} \geq \frac12 = \pr(F_2 \mbox{ is connected}).\]
Thus we may assume from now on that $t \geq 3$.
\medskip  

  For each integer $k$ with $1 \leq k \leq t$ let $\cF_t^k$ be the set of forests in $\cF_t$ with $k$ components.  We shall show that for each such $k$
\begin{equation} \label{eqn.massk}
  \mass(\cF_t^k) \leq  N (\frac{n}{t})^{t-2}\cdot |\cF_t^k|.
\end{equation}
  Summing over $k$ will then give
\[  \mass(\cF_t) \leq  N (\frac{n}{t})^{t-2} \cdot |\cF_t|, \]
  and so, using also~(\ref{eqn.KT2})
\[ \pr(\Fn \mbox{ is connected}) = \frac{\mass(\cT_t)}{\mass(\cF_t)} \geq \frac{|\cT_t|}{|\cF_t|} 
= \pr(F_t \mbox{ is connected}).\]
  This will complete the proof of~(\ref{eqn.RF2}) and thus of the lemma.  Hence it remains now to prove~(\ref{eqn.massk}).
\medskip

  Fix an integer $k$ with $1 \leq k \leq t$.
  Given a partition ${\bf U}= (U_1,\ldots,U_k)$ of $[t]$ into $k$ unordered sets,
  let $J=J({\bf U}) = \{i: |U_i| \geq 2\}$, and let $\cF({\bf U})$ be the set of forests in $\cF_t^k$
  such that the $U_i$ are the vertex sets of the $k$ component trees.
  For non-empty sets $U \subseteq [t]$,
  let $p(U)= \prod_{i \in U} n_i$ and $s(U) = \sum_{i \in U} n_i$.
  Observe that the mass of a forest is the product of the masses of its component trees,
  and a singleton component just gives a factor~1.
  Now fix a partition ${\bf U}= (U_1,\ldots,U_k)$ as above.
  
  If $J=\emptyset$ then $\mass(\cF({\bf U}))= 1 = |\cF({\bf U})|$.  Now suppose that $J \neq \emptyset$.  
   Then by~(\ref{eqn.KT2})
\begin{eqnarray*} 
   \mass(\cF({\bf U}))
  &=&
   \prod_{i \in J} p(U_i) \,\left(\frac{s(U_i)}{|U_i|} \right)^{|U_i|-2}  |U_i|^{|U_i|-2}\\
  & \leq & N \cdot \prod_{i \in J} \left(\frac{s(U_i)}{|U_i|} \right)^{|U_i|-2}
    \cdot \prod_{i \in J} |U_i|^{|U_i|-2}\\
  & = & N \cdot \prod_{i \in J} \left(\frac{s(U_i)}{|U_i|} \right)^{|U_i|-2}
    \cdot |\cF({\bf U})|.  
\end{eqnarray*}  

\noindent
  To handle the middle factor here, we can use Jensen's inequality, since $\log(x)$ is concave:
  we have
\begin{eqnarray*}  
 &&
   \log \, \prod_{i \in J} \left(\frac{s(U_i)}{|U_i|} \right)^{|U_i|-2}\\
 &=&
  (t-2) \sum_{i \in J} \frac{|U_i|-2}{t-2} \log \frac{s(U_i)}{|U_i|} \\
 & \leq &
  (t-2) \sum_{i \in J} \frac{|U_i|}{t} \log \frac{s(U_i)}{|U_i|} \hspace{.3in} \mbox{ since } |U_i| \leq t\\
  & \leq &
  (t-2) \sum_{i=1}^{k} \frac{|U_i|}{t} \log \frac{s(U_i)}{|U_i|} \\
 & \leq &
  (t-2) \log \left(\sum_{i=1}^{k} \frac{|U_i|}{t} \frac{s(U_i)}{|U_i|} \right) \;\;\; \mbox{ since log is concave}\\
 & = &
  (t-2) \log \frac{n}{t}.  
\end{eqnarray*}
  Hence in each case
\[  \mass(\cF({\bf U})) \leq N \left(\frac{n}{t} \right)^{t-2} |\cF({\bf U})|. \]
  So, summing over partitions ${\bf U}=(U_1,\ldots,U_k)$ of $[t]$,
\begin{eqnarray*}
   \mass(\cF_t^k) 
  & =&
    \sum_{{\bf U}=(U_1,\ldots,U_k)} \mass(\cF({\bf U}))\\
  & \leq &
   \sum_{{\bf U}=(U_1,\ldots,U_k)} N \left(\frac{n}{t} \right)^{t-2} |\cF({\bf U})|\\
  & = &
   N \left(\frac{n}{t} \right)^{t-2} |\cF_t^k|.
\end{eqnarray*}
  This completes the proof of~(\ref{eqn.massk}), and thus the proof of Lemma~\ref{lem.1}.
\end{proofof}
\bigskip

\noindent
 To prove Lemma~\ref{lem.2a} we will use the standard inequality
\begin{equation} \label{standard}
  (1-\frac{j}{n})^{n-j} \geq e^{-j} \;\;\;\; \mbox{ for } 1 \leq j < n.
\end{equation}
  [To see this, fix $j$ and let $g(x)=(x-j) \log (1-\frac{j}{x})$ for $x>j$. Then
\[ g'(x) = (x-j)(\frac1{x-j}-\frac1{x})+ \log(1-\frac{j}{x}) = \frac{j}{x} +  \log(1-\frac{j}{x}) <0, \]
  and so $g(n)$ is decreasing for $n>j$.  But $g(n) \to e^{-j}$ as $n \to \infty$, so $g(n)>e^{-j}$ for each $n>j$.]   
  \medskip

\begin{proofof}{Lemma~\ref{lem.2a}} \hspace{.1in}
For a graph $G$ let $\frag(G)$ be the number of vertices in $G$ less the number of vertices in a largest component; and for integers $n$ and $j$ with $1 \leq j < n$ let $f(n,j)$ be the number of forests $F$ on $[n]$ with $\frag(F)=j$. 
By~(\ref{standard}), for $1 \leq j < n/2$
\begin{eqnarray*} 
  f(n,j) & = &
  \binom{n}{j} |\cF_j| (n-j)^{n-j-2}\\
 & = &
   n^{n-2} \cdot \frac{|\cF_j|}{j!} \cdot \frac{(n)_j}{n^j} (1-\frac{j}{n})^{n-j-2}\\
  & \geq &
  n^{n-2} \cdot \frac{|\cF_j|}{j! \, e^j} \cdot \frac{(n)_j}{n^j} (1-\frac{j}{n})^{-2}.
\end{eqnarray*}
Now consider just $j \leq 2$ and let $n \geq 5$. Then $\frac{(n)_j}{n^j} (1-\frac{j}{n})^{-2} \geq 1$, so
\[ \frac{|\cF_n|}{n^{n-2}} > \sum_{j=0}^{2} \frac{|\cF_j|}{j! \, e^j}
   = 1+ \frac1{e} + \frac{2}{2! \, e^2} \approx 1.5032 \approx e^{0.4076}.\]
It is easy to check that this holds also for $n=2, 3$ and $4$; so
\[  \pr(F_n \mbox{ is connected}) < e^{-2/5} 
\;\; \mbox{ for each } n \geq 2,\]
as required.
\end{proofof}


\section{Concluding Remarks}
\label{sec.concl}

We can easily improve on Lemma~\ref{lem.2a} by pushing the proof further and doing some checking.
  
\begin{lemma} \label{lem.2}
If we set $\alpha =0.48$ then for each $n=2,3,\ldots$
\[ \pr(F_n \mbox{ is connected}) < e^{-\alpha}.\]
\end{lemma}

\begin{proof} 
  \hspace{.1in}
  It is straightforward to check that $\frac{(n)_j}{n^j} (1-\frac{j}{n})^{-2} \geq 1$ for each $j \leq 6$ and $n>12$.
  Hence, arguing as in the proof of Lemma~\ref{lem.2a}, for $n>12$
\[ \frac{|\cF_n|}{n^{n-2}} > \sum_{j=0}^{6} \frac{|\cF_j|}{j! \, e^j}
  \approx 1.6167 \approx e^{0.4804} > e^{0.48}.\]
This holds also for $2 \leq n \leq 12$: to check this we may for example use~\cite{oeis} for the values $|\cF_j|$ for $j \leq 12$.
\end{proof}
\medskip

  Lemma~\ref{lem.2} allows us to strengthen Theorem~\ref{thm.newa} as follows: with the same premises, if we set $\alpha=0.48$ then
\begin{equation} \label{eqn.thmnew}
 \pr(R_n \mbox{ is connected}) \geq \min_{\alpha n \leq t \leq n} \pr(F_{t} \mbox{ is connected}).
\end{equation}

  It is well known (see for example Flajolet and Sedgewick~\cite{fs09} Section II.5.3)
  that $\sum_{j \geq 1} \frac{|\cT_j|}{j! \, e^j} = \frac12$
  and so by the exponential formula $\sum_{j \geq 0} \frac{|\cF_j|}{j! \, e^j} = e^{\frac12}$.
  We could expect with more work to increase the value $\alpha=0.48$ in~(\ref{eqn.thmnew})
  to nearer $\frac12$ -- but can we go all the way to $\frac12$?
  
Perhaps $\pr(F_n \mbox{ is connected})$ is increasing from $n = 4$ onwards?  (For $n=1,\ldots,6$ the values of the probability are $1$, $\frac12$, $\frac37 \approx 0.4286$, $\frac{8}{19} \approx 0.4211$, $\frac{125}{291} \approx 0.4295$, $\frac{1296}{2932} \approx 0.4420$ (to 4 decimal places), with minimum at $n=4$.)  In that case, we would have $\pr(F_n \mbox{ is connected}) \leq e^{-\frac12}$ for each $n \geq 2$; and we could improve the bounds in Theorem~\ref{thm.newa} and in~(\ref{eqn.thmnew}) to
\begin{equation} \label{conj.concl1}
  \pr(R_n \mbox{ is connected}) \geq \pr(F_{\lceil n/2 \rceil} \mbox{ is connected}) \;\;\; \mbox{ for all } n \geq 7,
\end{equation}
which is getting closer to Conjecture~(\ref{conj2}).
Let us re-state the above question as a final conjecture.  
\begin{conjecture} \label{conj.incr}
$\pr(F_n \mbox{ is connected})$ is increasing for $n \geq 4$.
\end{conjecture}

In work in progress jointly with Xena Cologne-Brookes, 
we have shown using standard analytic methods (following a suggestion from a referee) that $\pr(F_n \mbox{ is connected})$ is strictly increasing for $n$ sufficiently large, which shows that the inequality~(\ref{conj.concl1}) holds for $n$ sufficiently large.  The aim is to establish the full Conjecture~\ref{conj.incr}, and thus the full inequality~(\ref{conj.concl1}), though the proof seems to depend on careful analytic estimates together with checking for many small values of $n$  (and thus to be of a different nature from the combinatorial proofs in this paper).
\bigskip

\noindent
{\bf Acknowledgements}
  I am grateful to Kostas Panagiotou for pointing out a problem with an earlier version of a proof; and to the referees for helpful comments, and to one referee in particular for suggesting how to use 
analytic methods to improve on Lemma~\ref{lem.2}.



\begin{thebibliography}{14}

\bibitem{amr2012} L.~Addario-Berry, C. McDiarmid and B. Reed.
  \newblock Connectivity for bridge-addable monotone graph classes,
  {\em Combinatorics, Probability and Computing}
  {\bf 21} (2012) 803 -- 815.
   
  \bibitem{bbg2008}
  P.~Balister, B. Bollob{\'a}s and S. Gerke,
  {Connectivity of addable graph classes}. {\em J. Combin. Th. B} {\bf 98} (2008) 577 -- 584.
  
\bibitem{bbg2010}
  P.~Balister, B. Bollob{\'a}s and S. Gerke,
  Connectivity of random addable graphs, {\em Proc. ICDM 2008} No 13 (2010) 127 -- 134.
  
\bibitem{cp2016}
G. Chapuy and G. Perarnau,  
Connectivity in bridge-addable graph classes: the McDiarmid-Steger-Welsh conjecture, arXiv:1504.06344, 2015.

  
\bibitem{fs09}
  P.~Flajolet and R.~Sedgewick,
  {\em Analytic Combinatorics,}
  Cambridge University Press, 2009.
  
\bibitem{kp2013}
  M.~Kang and K. Panagiotou,
  On the connectivity of random graphs from addable classes,
  {\em J. Combinatorial Theory B} {\bf 103} (2013) 306 -- 312.

\bibitem{lovasz}
  L.~Lov\'asz, {\em Combinatorial Problems and Exercises}, 2nd ed., North Holland, 1993.

\bibitem{cmcd2013}
  C.~McDiarmid,
  Connectivity for random graphs from a weighted bridge-addable class,
  {\em Electronic J Combinatorics} {\bf 19}(4) (2012) P53.
  
\bibitem{msw05}
  C.~McDiarmid, A.~Steger and D.~Welsh,
  Random planar graphs,
  {\em J. Combinatorial Theory B} {\bf 93} (2005) 187 -- 206.

\bibitem{msw06}
  C.~McDiarmid, A.~Steger and D.~Welsh,
  Random graphs from planar and other addable classes,
  {\em Topics in Discrete Mathematics} (M. Klazar, J. Kratochvil,
  M. Loebl, J. Matousek, R. Thomas, P. Valtr, Eds.),
  Algorithms and Combinatorics 26, Springer, 2006, 231 -- 246.

\bibitem{moon70}
  J.W.~Moon,
 {\em Counting labelled trees},
 {Canadian Mathematical Monographs} {\bf 1}, 1970.

\bibitem{norin2013}
  S.~Norin, Connectivity of addable classes of forests, private communication, 2013.

\bibitem{renyi59}
   A.~R\'enyi, 
   Some remarks on the theory of trees,
   {\em Publications of the Mathematical Institute of the
   Hungarian Academy of Sciences} {\bf 4} (1959) 73 -- 85.
 
\bibitem{oeis}
  The On-Line Encyclopedia of Integer Sequences, A001858, 
   November 2013, http://oeis.org.

 
\end{thebibliography}
\end{document}